\documentclass[a4paper,10pt]{amsart}
\usepackage[utf8]{inputenc}
\usepackage{amsfonts, amsmath, amssymb, amsthm}
\usepackage{hyperref}
\usepackage{tikz}
\usetikzlibrary{matrix,arrows,trees}

\swapnumbers
\newtheorem{definition}{Definition}[section]

\newtheorem{theorem}[definition]{Theorem}
\newtheorem{proposition}[definition]{Proposition}
\newtheorem{lemma}[definition]{Lemma}
\newtheorem{corollary}[definition]{Corollary}

\newtheorem{remark}[definition]{Remark}
\newtheorem{claim}[definition]{Claim}

\title{On the transfer reducibility of certain Farrell--Hsiang groups}
\author{Christoph Winges}
\address{Christoph Winges\\
	Westf\"alische Wilhelms-Universit\"at M\"unster,
		Mathematisches Institut\\
	Einsteinstr.~62 \\ 48149 M\"unster \\ Germany}
\email{christoph.winges@wwu.de}

\newcommand{\KK}{\mathbb{K}}
\newcommand{\NN}{\mathbb{N}}
\newcommand{\RR}{\mathbb{R}}
\newcommand{\ZZ}{\mathbb{Z}}

\newcommand{\cA}{\mathcal{A}}
\newcommand{\cC}{\mathcal{C}}
\newcommand{\cD}{\mathcal{D}}
\newcommand{\cE}{\mathcal{E}}
\newcommand{\cF}{\mathcal{F}}
\newcommand{\cP}{\mathcal{P}}
\newcommand{\cR}{\mathcal{R}}
\newcommand{\cS}{\mathcal{S}}
\newcommand{\cV}{\mathcal{V}}

\renewcommand{\epsilon}{\varepsilon}
\renewcommand{\theta}{\vartheta}
\renewcommand{\phi}{\varphi}

\newcommand{\abs}[1]{\lvert #1 \rvert}
\newcommand{\bound}{\mathsf{bd}}
\newcommand{\catscplx}{\mathsf{SCplx}}
\newcommand{\cyc}{\cC yc}
\newcommand{\depth}{d}
\newcommand{\dress}{\cD r}
\newcommand{\fhbdd}{Dress--Farrell--Hsiang group of bounded depth}
\newcommand{\gscplx}[1]{$#1$-simplicial complex}
\newcommand{\reso}{resolution}
\newcommand{\srd}{set of resolution data}
\newcommand{\transport}[2]{\mathsf{Tr}_{#1}(#2)}

\DeclareMathOperator{\id}{id}
\DeclareMathOperator{\rank}{rk}
\DeclareMathOperator{\vcd}{vcd}

\begin{document}

\begin{abstract}
 We show how the existing proof of the Farrell--Jones Conjecture for virtually poly-$\ZZ$-groups
 can be improved to rely only on the usual inheritance properties in combination
 with transfer reducibility as a sufficient criterion for the validity of the conjecture.
\end{abstract}

\maketitle

\section{Introduction}

The Farrell--Jones Conjecture predicts that a certain \emph{assembly map}
\begin{equation*}
 \alpha^G_{\cV\cC yc} \colon H^G_n(E_{\cV\cC yc}G;\KK^{-\infty}_\cA) \to K_n(\cA *_G G/G)
\end{equation*}
is an isomorphism for all discrete groups $G$ and small additive $G$-categories $\cA$;
there is also an $L$-theoretic version of the conjecture which replaces the non-connective
$K$-theory spectrum of \cite{MR802790} by Ranicki's ultimate lower $L$-groups (see e.g. \cite[\S\ 17]{MR1208729}, \cite[Sec.~4]{MR1341817}).
These conjectures have received a lot of attention since their introduction in \cite{MR1179537} due to their intimate relation
with other prominent conjectures such as the Borel and Novikov Conjectures.
See \cite{MR2181833} for a survey.

While the conjectures are still wide open in general, substantial progress has been made
on the question in which special cases the conjectures hold. Among the most notable
classes of examples, one finds hyperbolic \cite{MR2385666} and ${\rm CAT}(0)$-groups \cite{MR2993750, MR2869063},
virtually poly-$\ZZ$-groups \cite{MR3164984}, lattices in virtually connected Lie groups \cite{MR3164984, KLR-FJCForLatticesInLieGroups},
a large number of linear groups \cite{MR3210177}, and solvable groups \cite{W-FJCSolvableGroups}.

Normally, the proofs can be broken down into several steps. Starting from the most general case,
one uses certain inheritance properties of the conjectures to reduce the proof to simpler instances.
These are then dealt with by proving that the groups under consideration satisfy the
assumptions of an abstract criterion which has been independently shown to imply the conjecture.

These criteria include two prototypical examples. First, there is the notion of
\emph{transfer reducibility} which was used to prove the case of hyperbolic groups
and then generalised to also cover ${\rm CAT}(0)$-groups. Second, the property of
\emph{being a Farrell--Hsiang group} (an abstraction of the arguments employed by
Farrell and Hsiang in \cite{MR0482771, MR631942, MR704219} and further exploited by Quinn \cite{MR2826431})
was considered to obtain proofs for virtually poly-$\ZZ$-groups.

The goal of the present article is to show that the $K$- and $L$-theoretic
Farrell--Jones Conjectures for virtually poly-$\ZZ$-groups \cite[Thm.~1.1]{MR3164984} can be
deduced relying entirely on transfer reducibility as a sufficient criterion, bypassing
any use of the Farrell--Hsiang method (i.e., the results of \cite{MR2957625}).
This illustrates that being transfer reducible is not a concept inherently
concerned with non-positive curvature conditions, as the original examples
of transfer reducible groups, namely hyperbolic and ${\rm CAT}(0)$-groups, might suggest.

As a consequence of our results, the proofs of the Farrell--Jones Conjecture for
lattices in virtually connected Lie groups \cite[Thm.~1.2]{MR3164984}, \cite{KLR-FJCForLatticesInLieGroups}
also become independent of the Farrell--Hsiang method (e.g. the proof
of Theorem 1.2 in \cite{MR3164984} only requires the validity of the conjecture for
${\rm CAT}(0)$-groups as additional input).

It should be pointed out that the results of this article cannot be considered a simplification
of the existing proofs, as the arguments that go into the verification of
the Farrell--Hsiang condition still have to be employed. However, it does
serve the purpose of unifying the existing proofs; the results of \cite{MR2957625}
are not needed anymore. 

Virtually cyclic groups form a notable exception to the slogan that
``Farrell--Hsiang groups are transfer reducible''.
The reduction to the family of (possibly infinite) hyperelementary groups
presented in \cite[Prop.~3.1.1]{MR2826431} and \cite[Sec.~8]{MR3164984}
cannot be obtained with the methods of this article.

We proceed as follows: In \S \ref{sec_resolutions} and \S \ref{sec_dfh},
we formulate a strengthening of the Farrell--Hsiang condition and prove
that all groups which satisfy this stronger condition are transfer reducible
in a very strict sense. Once this has been done, we give a quick review of 
the structure of the proof of \cite[Thm.~1.1]{MR3164984} in \S \ref{sec_overview}.
This serves the purpose of singling out all instances of the Farrell--Hsiang condition
that appear in the proof. 
In the remaining sections \S \ref{sec_crystal} and \S \ref{sec_affine},
we will present proofs that the classes of groups isolated in \S \ref{sec_overview}
all satisfy our stronger version of the Farrell--Hsiang condition.
The appendix reviews a theorem of Oliver \cite{MR0375361} concerning fixed-point free actions of finite groups
on finite, contractible complexes which is required for the discussion in \S \ref{sec_dfh}.

The author expects that the results proved in this article will also have
applications in the algebraic $K$-theory of spaces; these will be presented elsewhere.

\textbf{Acknowledgements.}
This article contains parts of the author's PhD thesis written under the supervision
of Arthur Bartels, whose support and advice are gratefully acknowledged.
Daniel Kasprowski provided helpful comments on an earlier version of this article.
The author was financially supported by CRC 878 {\it Groups, Geometry and Actions} of the German Research Foundation (DFG).

\section{Resolving fixed points of group actions on simplicial complexes}\label{sec_resolutions}

Our strategy is not to come up with entirely new proofs whenever we wish
to replace an invocation of the Farrell--Hsiang method.
Instead, we will relate (a variant of) the Farrell--Hsiang condition
to the notion of transfer reducibility, and then improve the existing verifications of the
Farrell--Hsiang condition.

The major difference between the proofs of the Farrell--Jones Conjecture in \cite{MR2385666} and \cite{MR2957625}
lies in the construction of the transfer maps. While the proof in \cite{MR2385666} exploits the existence
of a compact transfer space, the Farrell--Hsiang method relies on a discrete $G$-set
for the transfer and uses an algebraic result due to Swan \cite[Cor.~4.2(c) \& Prop.~1.1]{MR0138688} as additional input.

In this section, we prove a result on the space level which is analogous to
Swan's induction theorem. This will enable us to produce appropriate transfer spaces.

The objects of interest are \gscplx{G}s; these are (abstract) simplicial complexes $X$
equipped with a $G$-action such that whenever a group element $g$ fixes a simplex $x = \{ x_0, \dots, x_n \}$
in $X$, then $g x_i = x_i$ for all $0 \leq i \leq n$. Note that for such a complex,
the entire group action is encoded in the action of $G$ on the set of $0$-simplices $X_0$.
The geometric realisation of a \gscplx{G}\ is a $G$-CW-complex.

Recall that the \emph{transport groupoid} $\transport{G}{T}$ of a $G$-set $T$
is the groupoid whose objects are the elements of $T$ and whose morphisms $g \colon t \to t'$ are group elements 
$g \in G$ such that $gt = t'$.

\begin{definition}\label{def_SetOfResoultionData}
 Let $X$ be a \gscplx{G}. A \emph{\srd} $\cR$ for $X$ is a functor
 \begin{equation*}
  \cR \colon \transport{G}{X_0} \to \catscplx
 \end{equation*}
 to the category of simplicial complexes with the following properties:
 \begin{itemize}
  \item If $g \colon x \to x$ is an endomorphism in $\transport{G}{X_0}$
   and $\cR(g)(y) = y$ for some simplex $y \in \cR(x)$, then $\cR(g)$ fixes $y$ pointwise.
  \item The vertex sets $\{ \cR(x)_0 \}_{x \in X_0}$ are pairwise disjoint.
 \end{itemize}
\end{definition}

\begin{remark}\label{rem_SetOfResolutionData}
 Giving a \srd\ is a ``coordinate-free'' way of specifying for each vertex $x$ of $X$
 a \gscplx{G_x} $\cR(x)$ together with isomorphisms $\cR(x) \cong \cR(x')$ for vertices
 $x$ and $x'$ lying in the same $G$-orbit.
 
 To be more precise, if we pick some $G$-orbit $\overline{x}$ in $X_0$ and fix
 a representative $x_0 \in \overline{x}$, we can define an isomorphism
 \begin{equation*}
  \coprod_{x \in \overline{x}} \cR(x) \xrightarrow{\cong} G \times_{G_{x_0}} \cR(x_0)
 \end{equation*}
 by sending $y \in \cR(x)$ to $(g,\cR(g^{-1})y)$, where $g$ is given by the condition that $gx_0 = x$.
 This isomorphism becomes $G$-equivariant if we equip the domain with the $G$-action given by
 $g \cdot y := \cR(g)(y)$.
\end{remark}

\begin{definition}\label{def_resolution}
 Let $X$ be a \gscplx{G} and $\cR$ a \srd\ on $X$. We define a simplicial complex
 $X[\cR]$ as follows: The set of vertices is given by $\coprod_{x \in X_0} \cR(x)_0$.
 A set $y = \{ y_0,\dots,y_n \}$ spans an $n$-simplex in $X[\cR]$ if the following holds:
 \begin{itemize}
  \item For every $x \in X_0$, the set $S_x(y)$ which contains those elements of $y$ which are vertices of $\cR(x)$ is a simplex in $\cR(x)$.
  \item The set $S(y) := \{ x \in X_0 \mid S_x(y) \neq \emptyset \}$ is a simplex in $X$.
 \end{itemize}
 The group $G$ acts on $X[\cR]$ by
 \begin{equation*}
  g \cdot \{y_0,\dots,y_n\} := \{ \cR(g)(y_0),\dots,\cR(g)(y_n) \}.
 \end{equation*}
 We call $X[\cR]$ the \emph{\reso\ of $X$ by $\cR$}.
\end{definition}

It is easy to check that $X[\cR]$ is indeed a \gscplx{G}.
In particular, a subgroup of $G$ can only appear as a stabiliser of $X[\cR]$
if it is a stabiliser group of some $\cR(x)$.
Let us also observe that any simplex $y \in X[\cR]$ can be partitioned into
\begin{equation*}
 y = \coprod_{x \in S(y)} S_x(y).
\end{equation*}
If $X$ is $n$-dimensional and there is some $k$ such that the dimension of $\cR(x)$ is at most $k$ for all $x$,
it follows that the dimension of $X[\cR]$ can be bounded by $(n + 1)(k + 1) - 1 = nk + n + k$.

If $\tau$ is a natural transformation of sets of resolution data $\cR \to \cR'$,
there is an induced $G$-equivariant simplicial map $X[\tau] \colon X[\cR] \to X[\cR']$
which sends $y \in \cR(x)_0$ to $\tau_x(y)$.

\begin{remark}\label{rem_AlternativeRealisation}
 Let $X$ be a \gscplx{G} and let $\cR$ be a \srd\ for $X$. Define $\abs{X,\cR}$ as the set of formal finite convex combinations
 \begin{equation*}
  \abs{X,R} := \left\{ \sum_{x \in X_0} \lambda_x \cdot \eta_x \mid \lambda_x \geq 0, \{ x \mid \lambda_x \neq 0 \} \in X, \sum_{x \in X_0} \lambda_x = 1, \eta_x \in \abs{\cR(x)} \right\},
 \end{equation*}
 Define
 \begin{equation*}
  F \colon \abs{X[\cR]} \to \abs{X,\cR}, \quad \sum_{y \in X[\cR]_0} \lambda_y \cdot y \mapsto \sum_{x \in X_0} \lambda_x \cdot \Big( \sum_{y \in \cR(x)_0} \frac{\lambda_y}{\lambda_x} \cdot y \Big),
 \end{equation*}
 where $\lambda_x := \sum_{y \in \cR(x)_0} \lambda_y$. Since we think of points in $\abs{X,\cR}$ as finite sums, we do not worry about the fact that $\frac{\lambda_y}{\lambda_x}$ is undefined when $\lambda_x = 0$.
 
 For $y \in X[\cR]_0$, let $x(y)$ denote the unique vertex $x \in X_0$ such that $y \in \cR(x)_0$. Given $x \in X_0$ and $\eta_x \in \abs{\cR(x)}$,
 write $\eta_x = \sum_{y \in \cR(x)_0} \eta_{x,y} \cdot y$. Then we can also define
 \begin{equation*}
  F' \colon \abs{X,\cR} \to \abs{X[\cR]}, \quad \sum_{x \in X_0} \lambda_x \cdot \eta_x \mapsto \sum_{y \in X[R]_0} \lambda_{x(y)}\eta_{x(y),y} \cdot y.
 \end{equation*}

 Then $F$ and $F'$ are mutually inverse bijections. This will turn out to be a more convenient model for $\abs{X[\cR]}$.
 In particular, the $\ell^1$-metric on $\abs{X[\cR]}$ induces a metric $d^1$ on $\abs{X,\cR}$ via $F'$.
\end{remark}

\begin{proposition}\label{prop_ResolutionHomotopyInvariance}
 Let $X$ be a \gscplx{G}\ and let $\cR$ and $\cR'$ be sets of resolution data for $X$.
 Suppose that $\tau \colon \cR \to \cR'$ is a natural transformation such that
 $\tau_x \colon \cR(x) \to \cR'(x)$ is a homotopy equivalence for all $x$.
 
 Then the induced map $X[\tau] \colon X[\cR] \to X[\cR']$ is a homotopy equivalence.
\end{proposition}
\begin{proof}
 Choose a homotopy inverse $f_x \colon \abs{\cR'(x)} \to \abs{\cR(x)}$ to $\abs{\tau_x}$ for every $x \in X_0$,
 and let $H_x \colon \abs{\cR(x)} \times [0,1] \to \abs{\cR(x)}$ and $H'_x \colon \abs{\cR'(x)} \times [0,1] \to \abs{\cR'(x)}$
 be homotopies witnessing $f_x \circ \abs{\tau_x} \simeq \id_{\abs{\cR(x)}}$ and $\abs{\tau_x} \circ f_x \simeq \id_{\abs{\cR'(x)}}$,
 respectively.
 
 Using the alternative description of $\abs{X[\cR]}$ from Remark \ref{rem_AlternativeRealisation}, we can define a (non-equivariant) map
 \begin{equation*}
  f \colon \abs{X,\cR'} \to \abs{X,\cR}, \quad \sum_x \lambda_x \cdot \eta_x \mapsto \sum_x \lambda_x \cdot f_x(\eta_x).
 \end{equation*}
 Similarly, there are induced homotopies  $H \colon \abs{X,\cR} \times [0,1] \to \abs{X,\cR}$ and
 $H' \colon \abs{X,\cR'} \times [0,1] \to \abs{X,\cR'}$, and it is easy to check that these witness
 $f \circ \abs{X[\tau]} \simeq \id_{\abs{X,\cR}}$ and $\abs{X[\tau]} \circ f \simeq \id_{\abs{X,\cR'}}$, respectively.
\end{proof}

\begin{corollary}\label{cor_ResolutionByContractibleComplexes}
 Let $X$ be a \gscplx{G} and $\cR$ a \srd\ on $X$. Suppose that $\cR(x)$ is contractible for all $x \in X_0$.
 Then the canonical map $X[\cR] \to X$ is a homotopy equivalence. \qed
\end{corollary}

We can utilise this construction to reduce the size of the stabilisers of a \gscplx{G} incrementally.
Specifically, we will give an answer to the question of what the smallest possible stabilisers of a finite group action
on a finite, contractible complex are.

\begin{definition}\label{def_Families}
 Let $\cyc$ denote the family of finite cyclic groups.
 For a given prime $p$, we let
 \begin{equation*}
  \begin{split}
   \cyc_p := \{ H \mid &\text{ There is an extension  } 1 \to P \to H \to C \to 1 \text{ such that } \\ &\text{\quad$P$ is a finite $p$-group and } C \in \cyc. \}
  \end{split}
 \end{equation*}
 denote the class of groups which are \emph{cyclic mod $p$}.
 
 Finally, we call
 \begin{equation*}
  \begin{split}
   \dress := \{ G \mid &\text{ There is an extension  } 1 \to H \to G \to Q \to 1 \text{ such that} \\ &\quad H \in \cyc_p \text{ and } Q \text{ is a finite $q$-group for some primes $p$ and $q$.} \}
  \end{split}
 \end{equation*}
 the \emph{Dress family}.
\end{definition}

\begin{definition}\label{def_depth}
 Let $G$ be a finite group. Define the \emph{depth of $G$} to be
 \begin{equation*}
  \begin{split}
   \depth(G) := \sup \{ n \mid &\text{ There is a properly descending chain of subgroups } \\& \quad G_1 \varsupsetneq G_2 \varsupsetneq \dots \varsupsetneq G_n \text{ in } G. \}.
  \end{split}
 \end{equation*}
\end{definition}

Observe that it is easy to find upper bounds for the depth of a finite group.
If $\abs{G} = p_1^{k_1} \dots p_r^{k_r}$ is the prime factorisation of the order of $G$,
the depth of $G$ cannot exceed $k_1 + \dots + k_r$.

\begin{theorem}[Oliver]\label{thm_OliversTheorem}
 There is a monotonely increasing, affine linear function $\bound \colon \NN_+ \to \NN_+$
 such that for every finite group $G \notin \dress$, there is a finite, contractible \gscplx{G}\ $X$ with $X^G = \emptyset$
 whose dimension is bounded by $\bound(\depth(G))$.
\end{theorem}
\begin{proof}
 Excluding the dimension bound, this is stated as one of the main results of \cite{MR0375361}.
 Nevertheless, the proof given by Oliver can be seen to provide the claimed bound.
 A short review of the proof which makes this explicit can be found in the appendix.
\end{proof}

% \begin{remark}
%  Owing to Theorem \ref{thm_OliversTheorem}, the complement of $\dress$ has become known as the class of \emph{Oliver groups},
%  and this terminology seems to be fairly common in the literature on transformation groups.\marginpar{Provide references?}
%  In this sense, ``Dress = non-Oliver''. In some places, the groups in $\dress$ are also called ``groups
%  admitting an isthmus series''\marginpar{Bak}.
% \end{remark}

\begin{corollary}\label{cor_OliversTheorem}
 For every finite group $G$, there is a finite, contractible \gscplx{G}\ $X$ whose stabilisers lie in $\dress$
 and whose dimension is bounded by
 \begin{equation*}
  \sum_{\emptyset \neq M \subset \{ 1, \dots, \depth(G) \} } \prod_{m \in M} \bound(m) \leq 2^{\depth(G)} \cdot \bound(\depth(G))^{\depth(G)}.
 \end{equation*}
\end{corollary}
\begin{proof}
 We can assume without loss of generality that $G \notin \dress$.
 Theorem \ref{thm_OliversTheorem} asserts the existence of a finite, contractible \gscplx{G} $X'$
 which does not have a global fixed point and whose dimension is bounded by $\bound(\depth(G))$.
 By induction, there exists for every vertex $x$ of $X'$ a finite, contractible \gscplx{G_x}\ $\cR(x)$
 whose stabilisers lie in $\dress$ and which satisfies the claimed dimension bound.
 Picking one such complex for a representative of each $G$-orbit gives rise to
 a \srd\ $\cR$ on $X'$ (see Remark \ref{rem_SetOfResolutionData}). Then $X := X'[\cR]$ has the desired properties.
\end{proof}

The \gscplx{G} $X$ from Corollary \ref{cor_OliversTheorem} will typically not be a model
for $E_{\dress}G$. See \cite[bottom of p.~93]{MR0423390} for a proof of this.

\section{A variant of the Farrell--Hsiang condition}\label{sec_dfh}

We are now ready to formulate our strengthening of the Farrell--Hsiang condition
and to relate it to the property of being transfer reducible.
Whenever we speak about generating sets of groups,
we assume these to be symmetric for convenience.

\begin{definition}\label{def_fhbdd}
 Let $G$ be a group and $S$ a finite generating set for $G$. Let $\cF$ be a family of subgroups of $G$.
 
 Call $(G,S)$ a \emph{\fhbdd\ with respect to $\cF$} if there exist $N \in \NN$ and $B \in \NN$ such that for every $\epsilon > 0$ there are
 \begin{itemize}
  \item an epimorphism $\pi \colon G \twoheadrightarrow F$ to a finite group with depth $\depth(F) \leq B$ and
  \item for every subgroup $D \leq F$ with $D \in \dress$ a \gscplx{G}\ $E_D$ of dimension at most $N$
   whose isotropy groups lie in $\cF$, and a $\overline{D} := \pi^{-1}(D)$-equivariant map
   $f_D \colon G \to E_D$ such that $d^{\ell^1}(f_D(g),f_D(g')) \leq \epsilon$ whenever $g^{-1}g' \in S$.
 \end{itemize}
 
 We say that $(G,S)$ is \emph{combinatorially transfer reducible with respect to $\cF$} if there exists $\nu \in \NN$
 such that for every $\epsilon > 0$ there are
 \begin{itemize}
  \item a finite, contractible \gscplx{G}\ $X$,
  \item a \gscplx{G}\ $E$ of dimension at most $\nu$ whose isotropy groups lie in $\cF$ and
  \item a map $f \colon X \to E$ which is $S$-equivariant up to $\epsilon$, i.e., such that
   \begin{equation*}
    d^{\ell^1}(sf(x),f(sx)) \leq \epsilon
   \end{equation*}
   holds for all $s \in S$ and $x \in X$.
 \end{itemize}
\end{definition}

Clearly, the notion of being a \fhbdd\ is a strengthening of the Farrell--Hsiang condition (\cite[Def.~1.1]{MR2957625}, \cite[Def.~2.14]{MR3164984}, \cite[Thm.~C]{AB-OnProofsOfTheFJC});
combinatorial transfer reducibility is a very special case of the more general notions of transfer reducibility
explained in \cite[Thm.~ A \& Thm.~B]{AB-OnProofsOfTheFJC} (see also \cite[Def.~1.8]{MR2993750}, \cite[Def.~0.4]{MR2967054}).

An immediate consequence of Corollary \ref{cor_OliversTheorem} is the following variant of the result
that finite groups satisfy the isomorphism conjecture with respect to the family of hyperelementary subgroups \cite[Thm.~2.9 \& Lem.~4.1]{MR2328114}:

\begin{corollary}
 Every finite group is combinatorially transfer reducible with respect to $\dress$.
 In particular, the assembly map
 \begin{equation*}
  \alpha^G_{\dress} \colon H^G_n(E_{\dress}G;\KK^{-\infty}_\cA) \to K_n(\cA *_G G/G)
 \end{equation*}
 is an isomorphism for every finite group $G$ and every small additive $G$-category $\cA$.
 \qed
\end{corollary}

The more important observation is the following.

\begin{theorem}\label{thm_fhbddAreTR}
 Let $G$ be a group and $S$ a finite generating set for $G$. Let $\cF$ be a family of subgroups of $G$.
 If $(G,S)$ is a \fhbdd\ with respect to $\cF$, then it is combinatorially transfer reducible with respect to $\cF$.
\end{theorem}

\begin{proof}
 Let $N$ and $B$ be the natural numbers whose existence is asserted in the
 definition of a \fhbdd. Fix $\epsilon > 0$, and pick an appropriate epimorphism
 $\pi \colon G \twoheadrightarrow F$ to a finite group with $\depth(F) \leq B$.
 By Corollary \ref{cor_OliversTheorem}, there exists a finite, contractible \gscplx{F}\ $X$
 whose dimension is bounded by $\beta := 2^B \cdot \bound(B)^B$ and whose isotropy groups lie in $\dress$.
 We equip $X$ with a $G$-action by restricting the $F$-action along $\pi$.
 
 After picking a representative in each $G$-orbit, the set of vertices $X_0$ can be decomposed into transitive $G$-sets
 \begin{equation*}
  X_0 = \coprod_{i \in I} G/\overline{D}_i,
 \end{equation*}
 where $\overline{D}_i$ is the preimage under $\pi$ of some $D_i \leq F$ with $D_i \in \dress$.
 
 For each $i \in I$, pick a \gscplx{G}\ $E'_i$ whose dimension is at most $N$ and whose isotropy groups lie in $\cF$,
 as well as a $\overline{D}_i$-equivariant map $f'_i \colon G \to E'_i$ such that $d^{\ell^1}(f'_i(g),f'_i(g')) \leq \epsilon$
 whenever $g^{-1}g' \in S$. Define $E_i := G \times_{\overline{D}_i} E'_i$, and let
 \begin{equation*}
  f_i \colon G/\overline{D}_i \to E_i, \quad g\overline{D}_i \mapsto (g,f'_i(g^{-1})).
 \end{equation*}
 Then $f_i$ is well-defined and $S$-equivariant up to $\epsilon$.  
 As explained in Remark \ref{rem_SetOfResolutionData}, the collection $\{ E_i \}_{i \in I}$ determines a \srd\ $\cE$ for $X$.
 
 Using the identification $E_i \cong \coprod_{g\overline{D}_i} E'_i$, one checks easily that $f_i(g\overline{D}_i)$
 is a point in the copy of $E'_i$ based on $g\overline{D}_i$. Therefore, using Remark \ref{rem_AlternativeRealisation},
 we may define $f \colon X \to \abs{X[\cE]}$ as a map $f \colon X \to \abs{X,\cE}$ via
 \begin{equation*}
  \begin{split}
%    f\big( \sum_{x \in X_0} \lambda_x x \big) =
   f\big( \sum_{i \in I} \sum_{g\overline{D}_i \in G/\overline{D}_i} \lambda_{g\overline{D}_i} g\overline{D}_i \big)
   &:= \sum_{i \in I} \sum_{g\overline{D}_i \in G/\overline{D}_i} \lambda_{g\overline{D}_i} f_i(g\overline{D}_i) \\
   &= \sum_{i \in I} \sum_{g\overline{D}_i \in G/\overline{D}_i} \lambda_{g\overline{D}_i} (g,f'_i(g^{-1})).
  \end{split}
 \end{equation*}
 Recall that the metric $d^1$ on $\abs{X,\cE}$ is induced by the $\ell^1$-metric on $\abs{X[\cE]}$
 via the bijection $F'$ from Remark \ref{rem_AlternativeRealisation}.
 Consequently, for two points $\eta = \sum_x \lambda_x \cdot \eta_x$ and $\theta = \sum_x \mu_x \cdot \theta_x$ in $\abs{X,\cE}$ we have
 \begin{equation*}
  \begin{split}
   d^1(\eta,\theta)
   = d^{\ell^1}(F'(\eta),F'(\theta))
   = \sum_{y \in X[\cE]_0} \abs{ \lambda_{x(y)}\eta_{x(y),y} - \mu_{x(y)}\theta_{x(y),y)} }.
  \end{split}
 \end{equation*}
 Using the triangle inequality, we can bound this by
 \begin{equation*}
  \begin{split}
   \sum_{y \in X[\cE]_0} &\abs{ \lambda_{x(y)}\eta_{x(y),y} - \mu_{x(y)}\theta_{x(y),y)} } \\
   &\leq \sum_{y \in X[\cE]_0} \abs{ \lambda_{x(y)}\eta_{x(y),y} - \lambda_{x(y)}\theta_{x(y),y} } + \sum_{y \in X[\cE]_0} \abs{ \lambda_{x(y)}\theta_{x(y),y} - \mu_{x(y)}\theta_{x(y),y} } \\
   &= \sum_{x \in X_0} \big( \lambda_x \cdot \sum_{y \in \cE(x)_0} \abs{ \eta_{x,y} - \theta_{x,y} } \big) + \sum_{x \in X_0} \big( \abs{\lambda_x - \mu_x} \cdot \sum_{y \in \cE(x)_0} \theta_{x,y} \big) \\
   &= \sum_{x \in X_0} \lambda_x \cdot d^{\ell^1}(\eta_x,\theta_x) + \sum_{x \in X_0} \abs{\lambda_x - \mu_x}.
  \end{split}
 \end{equation*}
 We will now use this estimate to prove that the map $f$ is $S$-equivariant up to $\epsilon$.
 Let $s \in S$. Then we have
 \begin{equation*}
  \begin{split}
   d^1\big( s \cdot &f\big( \sum_{i \in I} \sum_{g\overline{D}_i \in G/\overline{D}_i} \lambda_{g\overline{D}_i} g\overline{D}_i \big),
    f\big( s \cdot \sum_{i \in I} \sum_{g\overline{D}_i \in G/\overline{D}_i} \lambda_{g\overline{D}_i} g\overline{D}_i \big) \big) \\
   &= d^1\big( \sum_{i \in I} \sum_{g\overline{D}_i \in G/\overline{D}_i} \lambda_{s^{-1}g\overline{D}_i} (g,f'_i(g^{-1}s)),
    \sum_{i \in I} \sum_{g\overline{D}_i \in G/\overline{D}_i} \lambda_{s^{-1}g\overline{D}_i} (g,f'_i(g^{-1})) \\
   &\leq \sum_{i \in I} \sum_{g\overline{D}_i \in G/\overline{D}_i} \lambda_{s^{-1}g\overline{D}_i} d^{\ell^1}((g,f'_i(g^{-1}s)),(g,f'_i(g^{-1})) \\
    &\qquad+ \sum_{i \in I} \sum_{g\overline{D}_i \in G/\overline{D}_i} \abs{\lambda_{s^{-1}g\overline{D}_i} - \lambda_{s^{-1}g\overline{D}_i}} \\
   &\leq \epsilon   
  \end{split}
 \end{equation*}
 since the $\ell^1$-metric on any \gscplx{G}\ is $G$-invariant and $s^{-1}gg^{-1} = s^{-1} \in S$. So $f$ is $S$-equivariant up to $\epsilon$.
 All isotropy groups of the \gscplx{G}\ $X[\cE]$ lie in $\cF$, and the dimension of $X[\cE]$ is bounded by $\nu := \beta N + \beta + N$.
 Observe that $\nu$ depends only on $N$ and $B$. This proves the theorem.
\end{proof}

\section{Overview of the Bartels--Farrell--L\"uck--Quinn-argument}\label{sec_overview}

To keep the exposition in this and the following two sections short,
we do not elaborate on a number of arguments which are carried out in full detail in \cite{MR3164984};
the reader is advised to keep a well-read copy of \cite{MR3164984} close-by.

We are now going to give an outline of the proof of Thm.~1.1 in \cite{MR3164984}.
There are two classes of groups which play a particularly prominent role :

A group $\Gamma$ is called \emph{crystallographic} if it contains a normal subgroup $A$
which is finitely generated and free abelian such that $A$ has finite index and equals
its own centraliser in $\Gamma$. The subgroup $A$ is unique; the \emph{rank} of $\Gamma$
is defined to be the rank of $A$, and equals the virtual cohomological dimension of $\Gamma$
(see \cite[Sec.~3.1]{MR3164984}).

The second important class of groups is that of \emph{special affine groups};
these are those groups $\Gamma$ for which there is an extension
\begin{equation*}
 1 \to \Theta \to \Gamma \to \Delta \to 1
\end{equation*}
and an action $\rho' \colon \Gamma \times \RR^n \to \RR^n$ by affine motions such that
the restriction of $\rho'$ to $\Theta$ is a cocompact, isometric and proper action, and
$\Delta$ is either infinite cyclic or infinite dihedral.
A special affine group $\Gamma$ is \emph{irreducible} if for every epimorphism $\Gamma \to \Gamma'$
onto a virtually finitely generated abelian group $\Gamma'$,
the virtual cohomological dimension of $\Gamma'$ is at most $1$.

The proof of the Farrell--Jones Conjecture for virtually poly-$\ZZ$-groups in \cite{MR3164984}
makes heavy use of a number of inheritance properties of the conjecture. In order to
keep the exposition short, we do not recall these here, but refer the reader instead
to \cite[Sec.~2.3]{MR3164984} for a quick overview.

Let $G$ be a virtually poly-$\ZZ$-group. Using the Transitivity Principle,
the proof can proceed by induction on the virtual cohomological dimension of $G$.
One may assume that $\vcd(G) \geq 2$.
Then there is an extension $1 \to G_0 \to G \xrightarrow{pr} \Gamma \to 1$ in which $G_0$ is either finite
or a virtually poly-$\ZZ$-group which satisfies $\vcd(G_0) \leq \vcd(G) - 2$,
and $\Gamma$ is a crystallographic or a special affine group.
If $V$ is a virtually cyclic subgroup of $\Gamma$, its preimage under $pr$ is a virtually
poly-$\ZZ$-group with $\vcd(pr^{-1}(V)) < \vcd(G)$. So one only needs to prove
the conjecture for $\Gamma$.

\begin{claim}\label{claim_FJCCrystal}
 Every crystallographic group satisfies the Farrell--Jones Conjecture.
\end{claim}

\begin{claim}\label{claim_FJCIrrSpecialAffine}
 Every irreducible special affine group satisfies the Farrell--Jones Conjecture.
\end{claim}

Assuming the two claims, we need only consider the case that $\Gamma$ is a special affine group
which is not irreducible. Pick an epimorphism $p \colon \Gamma \to \Gamma'$, where $\Gamma'$
is some virtually finitely generated abelian group with $\vcd(\Gamma') \geq 2$.
For any virtually cyclic subgroup $V$ of $\Gamma'$, the preimage $p^{-1}(V)$ is a virtually
poly-$\ZZ$-group with $\vcd(p^{-1}(V)) < \vcd(G)$. This leaves the final case:

\begin{claim}\label{claim_FJCvfga}
 Every virtually finitely generated abelian group satisfies the Farrell--Jones Conjecture.
\end{claim}

In fact, the three claims also rely on each other: The proof of Claims \ref{claim_FJCCrystal} and \ref{claim_FJCvfga}
is dealt with simultaneously as follows (see also \cite{MR2826431}, especially Proposition 2.4.1).
One proceeds by induction over the virtual cohomological dimension of a virtually finitely generated abelian group $\Gamma$
as well as the smallest order of a finite group $F$ which fits into an exact sequence $1 \to \ZZ^{\vcd(\Gamma)} \to \Gamma \to F \to 1$
(Call the order of the group $F$ the \emph{holonomy} of $\Gamma$).

One may assume that $\vcd(\Gamma) \geq 2$.  Since there is an epimorphism with finite kernel onto a crystallographic
group with the same virtual cohomological dimension, the Transitivity Principle serves to reduce the proof to the following claim:

\begin{claim}\label{claim_FJC_Crystal.vcd}
 Every crystallographic group $\Gamma$ with $\vcd(\Gamma) \geq 2$ satisfies the isomorphism conjecture with respect to
 the family of all subgroups $G \leq \Gamma$ which satisfy one of the following conditions:
 \begin{itemize}
  \item $\vcd(G) < \vcd(\Gamma)$.
  \item $\vcd(G) = \vcd(\Gamma)$ and the holonomy of $G$ is smaller than the holonomy of $\Gamma$.
 \end{itemize}
\end{claim}

This takes care of claims \ref{claim_FJCCrystal} and \ref{claim_FJCvfga}. Once this has been done,
the Transitivity Principle may be invoked another time to see that it suffices to show the following claim to finish the proof.

\begin{claim}\label{claim_FJC_IrrSpecialAffine.vfga}
 Every irreducible special affine group satisfies the isomorphism conjecture with respect to the family of virtually finitely generated abelian groups.
\end{claim}

The upshot of this discussion is that we will have to provide proofs for claims 
\ref{claim_FJC_Crystal.vcd} and \ref{claim_FJC_IrrSpecialAffine.vfga}
which avoid the use of the Farrell--Hsiang method.

\section{Transfer reducibility of crystallographic groups}\label{sec_crystal}

Let us fix the following notational conventions:
Recall that every crystallographic group $\Gamma$ fits into a short exact sequence
\begin{equation*}
 1 \to A \to \Gamma \xrightarrow{pr} F \to 1 
\end{equation*}
with $A$ a finitely generated, free abelian group
which equals its own centraliser, and $F$ a finite group.
Observe that there is always a canonical action of $F$ on $A$
by considering the conjugation actions of arbitrary lifts of elements of $F$
under $pr$.

Let $s \in \NN$. Then we denote by $A_s$ the quotient $A/sA$. Since $sA$ is also normal in $\Gamma$,
and we define $\Gamma_s$ to be the quotient $\Gamma/sA$. We let $\pi_s \colon \Gamma \twoheadrightarrow \Gamma_s$
be the projection map. Moreover, the epimorphism $pr$ induces a surjective homomorphism $pr_s \colon \Gamma_s \twoheadrightarrow F$
whose kernel is precisely $A_s$.

The normal subgroup $A$ is isomorphic to $\ZZ^n$, where $n$ is the rank of $\Gamma$. Consequently,
$A_s \cong \ZZ^n/s\ZZ^n \cong (\ZZ/s)^n$.
It follows that $\abs{\Gamma_s} = \abs{A_s} \cdot \abs{F} = s^n \cdot \abs{F}$.
Therefore, it suffices to bound the number of prime factors in $s$
(counted with their multiplicities) if we want to bound the depth of the finite quotient $\Gamma_s$.

We will make regular use of the following observation:

\begin{lemma}\label{lem_DressGroupsNF}
 Let $G$ be a finite group, and suppose there is a normal series $P' \unlhd H' \unlhd G$
 such that $P'$ is a $p$-group, $H'/P'$ is cyclic and $G/H'$ is a $q$-group, i.e.\ $G$ lies in $\dress$.
 
 Then there is a normal series $P \unlhd H \unlhd G$ such that $P$ is a $p$-group,
 $H/P$ is cyclic, $G/H$ is a $q$-group and neither $p$ nor $q$ divide the order
 of $H/P$. Moreover, $H$ is isomorphic to a semidirect product $P \rtimes H/P$,
 the subgroup $P$ is normal in $G$ and $G$ is $q$-hyperelementary mod $p$.
\end{lemma}
\begin{proof}
 Let $\pi_1 \colon H' \to H'/P'$ be the projection. Since $H'/P'$ is cyclic,
 there is a unique cyclic $p$-Sylow group $S_p$ in $H'/P'$; so $H'/P'$
 decomposes as a direct product $H'/P' \cong S_p \times C'$. Set $P := \pi_1^{-1}(S_p)$.
 Then $P$ is a normal $p$-group in $H'$, and the quotient $H'/P$ is isomorphic to $C'$.
 
 If $p = q$, we set $H := H'$ and are done. Otherwise, let $\pi_2 \colon H' \to H'/P \cong C'$
 be the projection, and let $S_q$ be the unique $q$-Sylow group of $C'$.
 We have an isomorphism $C' \cong C \times S_q$. Set $H := \pi_2^{-1}(C)$.
 
 We claim that $H$ is normal in $G$. Let $h \in H$, $g \in G$, and suppose that $ghg^{-1} \notin H$.
 Since $H'$ is normal in $G$, we know that $ghg^{-1} \in H'$. The order of $ghg^{-1}$
 equals the order of $h$, so we conclude from $p \neq q$ that no power of $q$ divides
 the order of $ghg^{-1}$. Since we assumed that $ghg^{-1} \notin H$, this element
 maps to a non-trivial element in $H'/H \cong S_q$; but then there must be some
 power of $q$ which divides the order of $ghg^{-1}$, which is a contradiction.
 
 Set $Q := G/H$. The order of this group is given by
 \begin{equation*}
  \abs{Q} = \frac{\abs{G}}{\abs{H}} = \frac{\abs{H'}\cdot\abs{G/H'}}{\abs{H}}
   = \frac{\abs{P}\cdot\abs{C}\cdot\abs{S_q}\cdot\abs{G/H'}}{\abs{C}\cdot\abs{P}}
   = \abs{S_q}\cdot\abs{G/H'},
 \end{equation*}
 so $P \unlhd H \unlhd G$ is the desired normal series.
 
 The Schur--Zassenhaus Theorem \cite[Thm.~9.3.6]{MR896269} states that $H$ is a semidirect product.
 While this also implies that $P$ is normal in $G$, we prove normality by hand.
 Let $x \in P$ and $g \in G$, and suppose that $gxg^{-1} \notin P$. Since $H$ is normal,
 we know that $gxg^{-1} \in H$. So $gxg^{-1}$ defines a non-trivial element in $H/P$.
 In particular, there is some prime $l \neq p$ which divides the order of $gxg^{-1}$;
 this is a contradiction since $\abs{gxg^{-1}} = \abs{x}$ is a $p$-power.
 Moreover, the kernel of the natural surjection $G/P \twoheadrightarrow G/H$ is isomorphic to $H/P$,
 so $G/P$ is $q$-hyperelementary.
\end{proof}

\begin{lemma}[{cf.\ \cite[Lem.~3.8]{MR3164984}}]\label{lem_fhbdd.z2}
 The group $\ZZ^2 \rtimes_{-\id} \ZZ/2$ is a \fhbdd\ with respect to $\cV\cC yc$ (relative to an arbitrary finite generating set).
\end{lemma}
\begin{proof}
 Set $\Gamma := \ZZ^2 \rtimes_{-\id} \ZZ/2$, and let
 \begin{equation*}
   1 \to \ZZ^2 \xrightarrow{i} \Gamma \xrightarrow{pr} \ZZ/2 \to 1
 \end{equation*}
 be the obvious short exact sequence. Let $d_\Gamma$ be the word metric
 with respect to some chosen finite generating set $S$ of $\Gamma$.
 The map $\text{ev} \colon \Gamma \to \RR^2$ which evaluates the natural $\Gamma$-action
 on $\RR^2$ at the point $0$ is a quasi-isometry, so we can find positive constants $C_1$ and $C_2$
 such that for all $\gamma_1, \gamma_2 \in \Gamma$
 \begin{equation*}
  d_{\text{euc}}(\text{ev}(\gamma_1),\text{ev}(\gamma_2)) \leq C_1 \cdot d_\Gamma(\gamma_1,\gamma_2) + C_2
 \end{equation*}
 holds. Let $\epsilon > 0$. Pick three pairwise distinct odd prime numbers $p_1$, $p_2$, $p_3$ which satisfy
 \begin{equation*}
  p_i \geq \frac{8 \cdot (C_1 + C_2)^2}{\epsilon^2}.
 \end{equation*}
 Set $s := p_1p_2p_3$, so $\Gamma_s := (\ZZ/s\ZZ)^2 \rtimes_{-\id} \ZZ/2$. This fits into an exact sequence
 \begin{equation*}
  1 \to (\ZZ/s)^2 \to \Gamma_s \xrightarrow{pr_s} \ZZ/2 \to 1.
 \end{equation*}
 Since $\abs{\Gamma_s} = 2 \cdot (p_1p_2p_3)^2$, the depth of $\Gamma_s$ is uniformly bounded.
 
 Let $D \leq \Gamma_s$ be a subgroup which lies in $\dress$. Then there is a normal series $Q_0 \unlhd D_0 \unlhd D$
 such that $Q_0$ is a $q_0$-group, $D/D_0$ is a $q_1$-group, and $D_0/Q_0$ is a cyclic group of order prime to $q_0$ and $q_1$.
 Assume without loss of generality that $p_3 \notin \{ q_0, q_1 \}$, and consider the projection $\pi \colon \Gamma_s \twoheadrightarrow \Gamma_{p_3}$.
 Then $\pi(D \cap (\ZZ/s)^2)$ is a cyclic subgroup of $(\ZZ/p_3)^2 \leq \Gamma_{p_3}$.
 
 If $\pi(D \cap (\ZZ/s)^2)$ is non-trivial, use \cite[Lem.~3.7]{MR3164984} to find a homomorphism $r \colon \ZZ^2 \to \ZZ$ such that the kernel
 of the mod $p_3$-reduction of $r$ equals $\pi(D \cap (\ZZ/s)^2)$, and $r_\RR = r \otimes_\ZZ \RR$ satisfies
 $d_{\text{euc}}(r_\RR(x_1),r_\RR(x_2)) \leq \sqrt{2p_3} \cdot d(x_1,x_2)$ for all $x_1,x_2 \in \RR$.
 Otherwise, let $r \colon \ZZ^2 \to \ZZ$ be the projection onto the first factor.
 
 Set $\overline{D} := \pi_s^{-1}(D)$. Then $r(\overline{D} \cap \ZZ^2) \subset p_3\ZZ$,
 and the argument proceeds precisely as in \cite[Lem.~3.8]{MR3164984} from here on.
\end{proof}

\begin{proposition}[{cf.\ \cite[Lem.~3.15]{MR3164984}}]\label{prop_fhbdd.crystalreducible}
 Let $\Gamma$ be a crystallographic group of rank $2$ which possesses a normal infinite cyclic subgroup.
 Let $\cF$ be the family of subgroups of $\Gamma$ which contains all virtually cyclic groups
 as well as all groups which do not surject onto $F$ under $pr$.
 
 Then $\Gamma$ is a \fhbdd\ with respect to $\cF$ (relative to an arbitrary finite generating set).
\end{proposition}
\begin{proof}
 We will have to use the following facts which are proved at the beginning of the proof of Lemma 3.15
 in \cite{MR3164984}:
 \begin{itemize}
  \item $A$ decomposes uniquely into a direct sum $Z_1 \oplus Z_2$ of two infinite cyclic and $F$-invariant subgroups.
  \item $F$ is either $\ZZ/2$ or $\ZZ/2 \times \ZZ/2$.
 \end{itemize}
 If $Z$ is either $Z_1$ or $Z_2$, this is a normal subgroup of $\Gamma$ (since it is $F$-invariant).
 Let $\hat{\xi}_Z \colon \Gamma \twoheadrightarrow \Gamma/Z$ and $\xi_Z \colon A \twoheadrightarrow A/Z$
 be the projection maps. Furthermore, there is an epimorphism $\hat{\mu}_Z \colon \Gamma/Z \twoheadrightarrow \Delta_Z$
 onto $\Delta_Z \in \{ \ZZ, \ZZ \rtimes_{-\id} \ZZ/2 \}$ since $\Gamma/Z$ is virtually cyclic; the kernel of $\hat{\mu}_Z$ is finite,
 and the restriction $\mu_Z$ of $\hat{\mu}_Z$ to $A/Z$ is injective. Set
 \begin{equation*}
  \begin{split}
   \hat{\nu}_Z &:= \hat{\mu}_Z \circ \hat{\xi}_Z \colon \Gamma \twoheadrightarrow \Delta_Z, \\
   \nu_Z &:= \mu_Z \circ \xi_Z \colon A \to A_Z.
  \end{split}
 \end{equation*}
 Subject to some choice of finite generating set of $\Gamma$, there is a word metric $d_\Gamma$ on $\Gamma$.
 Let $\text{ev}_Z \colon \Delta_Z \to \RR$ be the map which is given by evaluating the natural $\Delta_Z$-action
 on $\RR$ at $0$. Equip $\RR$ with the simplicial structure whose set of vertices is $\{\frac{n}{2} \mid n \in \ZZ \}$.
 
 Then there are positive constants $C_1$ and $C_2$ such that for every $F$-invariant, infinite cyclic subgroup $Z$ of $A$
 and all $\gamma_1$, $\gamma_2 \in \Gamma$ we have
 \begin{equation*}
  d^{\ell^1}(\text{ev}_Z \circ \hat{\nu}_Z(\gamma_1),\text{ev}_Z \circ \hat{\nu}_Z(\gamma_2)) \leq C_1 \cdot d_\Gamma(\gamma_1,\gamma_2) + C_2.
 \end{equation*}
  Fix $\epsilon > 0$ and choose two odd prime numbers $p_1$ and $p_2$ such that
 \begin{equation*}
  p_i \geq \frac{2 \cdot (C_1 + C_2)}{\epsilon}
 \end{equation*}
 holds. Let $s := p_1p_2$. Consider a subgroup $D \leq \Gamma_s$ which lies in $\dress$, and assume without loss
 of generality that $pr_s(D) = F$. Pick a normal series $Q_0 \unlhd D_0 \unlhd D$ such that
 $Q_0$ is a $q_0$-group, $D/D_0$ is a $q_1$-group and $D_0/Q_0$ is a cyclic group of order prime to $q_0$ and $q_1$.
 
 If $q_0 \in \{ p_1, p_2 \}$, say $q_0 = p_2$, consider the projection $\pi \colon \Gamma_s \twoheadrightarrow \Gamma_{p_1}$.
 Then $\pi(D)$ is hyperelementary, and the argument on page 352 of \cite{MR3164984} shows that $\pi(D) \cap A_{p_1}$ is cyclic.
 
 If $q_0 \notin \{ p_1,p_2 \}$, the group $D \cap A_s$ is $q_1$-hyperelementary. In case $q_1$ is one of $p_1$ and $p_2$,
 let us say that $q_1 = p_1$. Otherwise, $D \cap A_s$ is even cyclic. In both cases, let $\pi \colon \Gamma_s \twoheadrightarrow \Gamma_{p_2}$
 be the projection. Then $\pi(D) \cap A_{p_2}$ is cyclic.
 
 In all cases we have considered, we have been able to find a prime $p \in \{ p_1,p_2 \}$ such that
 $\pi(D) \cap A_p$ is cyclic, where $\pi \colon \Gamma_s \twoheadrightarrow \Gamma_p$ is the canonical projection.
 Note that $\pi \circ \pi_s = \pi_p$. Set $\overline{D} := \pi_s^{-1}(D)$.
 It follows that there is $j \in \{1,2\}$ such that
 \begin{equation*}
  \begin{split}
   \overline{D} \cap A
   &\subset \pi_s^{-1}(\pi^{-1}(\pi(D))) \cap A \\
   &= \pi_p^{-1}(\pi(D) \cap A_p) \\
   &\subset \pi_p^{-1}(\pi_p(Z_j)).
  \end{split}
 \end{equation*}
 Hence, $\xi_{Z_j}(\overline{D} \cap A) \subset p(A/Z_j)$. The remainder of the proof is as in \cite[Lem.~3.15]{MR3164984}.
\end{proof}

\begin{proposition}\label{prop_fhbdd.crystal.vcd}
 Let $\Gamma$ be a crystallographic group of rank at least $2$ which does not contain a normal infinite cyclic subgroup.
 Let $\cF$ be the family of all subgroups of $\Gamma$ whose virtual cohomological dimension is smaller than that of $\Gamma$
 or which do not surject onto $F$ under $pr$.
 
 Then $\Gamma$ is a \fhbdd\ with respect to $\cF$ (relative to an arbitrary finite generating set).
\end{proposition}

An important ingredient in the proof of \ref{prop_fhbdd.crystal.vcd} is the following mild generalisation of the second part of \cite[Prop.~2.4.2]{MR2826431}.

\begin{lemma}\label{lem_dressgroupsinfinitequotients}
 Let $\Gamma$ be an arbitrary crystallographic group.
 Let $p_1$ and $p_2$ be prime numbers which do not divide $\abs{F}$,
 and let $r > 0$ be a natural number. Set $s := p_1^rp_2^r$.
 Suppose that $D \leq \Gamma_s$ is in $\dress$ such that $pr_s(D) = F$ and $D \cap A_s$ is non-trivial.
 
 Then there is $i \in \{ 1,2 \}$ such that $\pi(D) \cap A/p_i^rA$ contains a non-trivial $F$-invariant cyclic subgroup,
 where $\pi$ is the canonical projection $\Gamma_s \twoheadrightarrow \Gamma_{p_i^r}$.
\end{lemma}
\begin{proof}
 Pick a normal series $Q_0 \unlhd D_0 \unlhd D$ such that $Q_0$ is a $q_0$-group, $D/D_0$ is a $q_1$-group,
 and $D_0/Q_0$ is a cyclic group whose order is coprime to $q_0$ and $q_1$.
 
 Suppose $q_0 \in \{ p_1, p_2 \}$, say $q_0 = p_1$. Consider the projection $\pi \colon \Gamma_s \twoheadrightarrow \Gamma_{p_1^r}$.
 Then $\pi(Q_0)$ is trivial, so $\pi(D)$ is a hyperelementary subgroup of $\Gamma_{p_1^r}$ which surjects with
 a non-trivial kernel onto $F$. Apply \cite[Prop.~2.4.2]{MR2826431} to obtain the claimed result.
 
 If $q_0 \notin \{ p_1,p_2 \}$, then $Q_0$ injects into $F$ via $pr_s$ and $D \cap A_s$ is $q_1$-hyperelementary.
 If also $q_1 \notin \{ p_1, p_2 \}$, then $D \cap A_s = D_0 \cap A_s$ is a non-trivial cyclic group.
 Let $p \in \{ p_1, p_2 \}$ be a divisor of the order of $D \cap A_s$, and let $\pi \colon \Gamma_s \twoheadrightarrow \Gamma_{p^r}$
 be the projection. Then $\pi(D) \cap A_{p^r}$ is a non-trivial cyclic group. Since it equals the kernel
 of the surjection $\pi(D) \twoheadrightarrow F$, it is also $F$-invariant.
 
 Now suppose $q_1 \in \{ p_1, p_2 \}$, say $q_1 = p_1$. Then $pr_s(D_0) = F$.
 
 In case $D_0 \cap A_s$ is non-trivial, the projection $\pi \colon \Gamma_s \twoheadrightarrow \Gamma_{p_2^r}$ maps
 $D \cap A_s$ to the non-trivial cyclic group $\pi(D_0 \cap A_s)$. This group is normal in $\pi(D) = \pi(D_0)$,
 and since $\pi(D_0)$ also surjects onto $F$, it is also $F$-invariant.
 
 In case $D_0 \cap A_s$ is trivial, $pr_s$ restricts to an isomorphism $D_0 \cong F$, and $D \cap A_s$ is a $q_1$-group
 isomorphic to $D/D_0$. It follows that $D \cong D_0 \times D/D_0$. The projection $\pi \colon \Gamma_s \twoheadrightarrow \Gamma_{p_1^r}$
 maps $D$ isomorphically to a subgroup of $\Gamma_{p_1^r}$; the kernel of the surjection $\pi(D) \twoheadrightarrow F$
 is a non-trivial $q_1$-group (isomorphic to $D/D_0$). In particular, there is a non-trivial, normal and cyclic subgroup of the kernel.
 This subgroup is then also normal in $\pi(D)$ because $\pi(D)$ is the direct product of the kernel and $F$;
 in particular, it is $F$-invariant. 
\end{proof}

\begin{proof}[Proof of Proposition \ref{prop_fhbdd.crystal.vcd}, cf.\ {\cite[Sec.~3.3]{MR3164984}}]
 Let $d_\Gamma$ be the word metric on $\Gamma$ with respect to some finite generating set $S$.
 Let $\text{ev} \colon \Gamma \to \RR^n$ be the map that is given by evaluating the natural $\Gamma$-action
 on $\RR^n$ at $0$. There are positive constants $C_1$ and $C_2$ such that
 \begin{equation*}
  d_{\text{euc}}(\text{ev}(\gamma_1),\text{ev}(\gamma_2)) \leq C_1 \cdot d_\Gamma(\gamma_1,\gamma_2) + C_2
 \end{equation*}
 for all $\gamma_1,\gamma_2 \in \Gamma$.
 
 Fix $\epsilon > 0$. Pick a simplicial structure on $\RR^n$ such that $\RR^n$ is a \gscplx{\Gamma}.
 Then there is $\delta > 0$ such that
 \begin{equation*}
  d_{\text{euc}}(x_1,x_2) \leq \delta \quad \Rightarrow \quad d^{\ell^1}(x_1,x_2) \leq \epsilon
 \end{equation*}
 for all $x_1, x_2 \in \RR^n$.
 
 Write $\abs{F} = 2^k \cdot l$ with $l$ a non-negative odd natural number and $k$ a non-negative natural number.
 Using Dirichlet's Theorem \cite[IV.4.1]{MR0344216}, pick two distinct prime numbers $p_1$ and $p_2$ such that
 \begin{equation*}
  \begin{split}
   p_i &\equiv -1 \text{ mod } 4l, \\
   p_i &\geq \frac{C_1 + C_2}{\delta}, \\
   p_i &\geq \abs{F}.
  \end{split}
 \end{equation*}
 In particular, $\abs{F}$ is coprime to both $p_1$ and $p_2$. Set $r := \phi(\abs{F})$, where $\phi$ is Euler's $\phi$-function. Then
 \begin{equation*}
  p_i^r \equiv 1 \text{ mod } \abs{F}.
 \end{equation*}
 Put $s := p_1^rp_2^r$. Note that the depth of $\Gamma_s$ can be uniformly bounded, independent of the choice of $p_1$ and $p_2$.
 
 Consider a subgroup $D \leq \Gamma_s$ which lies in $\dress$. We may assume that $pr_s(D) = F$.
 Since $\Gamma$ contains no normal infinite cyclic subgroup, it follows from \ref{lem_dressgroupsinfinitequotients}
 in combination with \cite[Prop.~2.4.2]{MR2826431} that $D \cap A_s$ is trivial.
 
 The proof continues as in \cite[Sec.~3.3]{MR3164984} from this point on to show that $\Gamma$ is a \fhbdd\ with respect to $\cF$.
\end{proof}

Propositions \ref{prop_fhbdd.crystalreducible} and \ref{prop_fhbdd.crystal.vcd} in conjunction with Lemma \ref{lem_fhbdd.z2}
prove Claim \ref{claim_FJC_Crystal.vcd} by virtue of the Transitivity Principle. As was explained in \S \ref{sec_overview},
Claims \ref{claim_FJCCrystal} and \ref{claim_FJCvfga} follow from this.

\section{Transfer reducibility of irreducible special affine groups}\label{sec_affine}

We shall now deal with Claim \ref{claim_FJC_IrrSpecialAffine.vfga}.

\begin{theorem}\label{thm_fhbdd.IrrSpecialAffine}
 Every irreducible special affine group $\Gamma$ is a \fhbdd\ with respect to the family of virtually finitely generated abelian groups.
\end{theorem}

The main technical ingredient for the proof of Theorem \ref{thm_fhbdd.IrrSpecialAffine}
is the following generalisation of \cite[Sec.~4.4]{MR3164984}:

\begin{lemma}\label{lem_IrrSpecialAffine.finitequotients}
 There is a natural number $B$ such that for all natural numbers $o, \nu$ there are $r,s \in \NN$
 such that $s \equiv 1 \text{ mod } o$ and for all $M \in GL_n(\ZZ)$ the following holds:
 \begin{enumerate}
  \item The order of $GL_n(\ZZ/s)$ divides $r$; in particular, the semidirect product $(\ZZ/s)^n \rtimes_{M_s} \ZZ/r$
   is defined, where $M_s$ denotes the reduction of $M$ modulo $s$. Let $\pi \colon (\ZZ/s)^n \rtimes_{M_s} \ZZ/r \to \ZZ/r$
   be the projection.
  \item The order of $(\ZZ/s)^n \rtimes_{M_s} \ZZ/r$ contains at most $B$ prime factors, counted with their multiplicities.
  \item All subgroups $G \leq (\ZZ/s)^n \rtimes_{M_s} \ZZ/r$ which lie in $\dress$ have one of the following properties:
   \begin{enumerate}
    \item There is some $\nu' \geq \nu$ which divides $s$ such that $\nu' \equiv 1 \text{ mod } o$ and $G \cap (\ZZ/s)^n \subset \nu'(\ZZ/s)^n$.
    \item $[\ZZ/r \colon \pi(G)] \geq \nu$.
   \end{enumerate}
 \end{enumerate}
\end{lemma}
\begin{proof}
 Let us first explain how to obtain the bound on the number of prime factors in the order of $(\ZZ/s)^n \rtimes_{M_s} \ZZ/r$.
 Suppose $s$ is a product of pairwise distinct prime numbers $s  = p_1 \dots p_k$. Then $\ZZ/s \cong \ZZ/p_1 \times \dots \ZZ/p_k$
 as rings. Since the diagram
 \begin{equation*}
  \begin{tikzpicture}
   \matrix (m) [matrix of math nodes, column sep=1.5em, row sep=1.5em,text depth=.5em, text height=1em]
   {M_n(\ZZ/s) & M_n(\ZZ/p_1) \times \dots \times M_n(\ZZ/p_k) \\
   \ZZ/s & \ZZ/p_1 \times \dots \times \ZZ/p_k \\};
   \path[->]
   (m-1-1) edge node[above]{$\cong$} (m-1-2) edge node[left]{$\text{det}_s$} (m-2-1)
   (m-1-2) edge node[right]{$\text{det}_{p_1} \times \dots \times \text{det}_{p_k}$} (m-2-2)
   (m-2-1) edge node[above]{$\cong$} (m-2-2);
  \end{tikzpicture}
 \end{equation*}
 commutes and $(\ZZ/p_1 \times \dots \ZZ/p_k)^* = (\ZZ/p_1)^* \times \dots \times (\ZZ/p_k)^*$,
 the top isomorphism induces an isomorphism
 \begin{equation*}
  GL_n(\ZZ/s) \xrightarrow{\cong} GL_n(\ZZ/p_1) \times \dots \times GL_n(\ZZ/p_k).
 \end{equation*}
 Define a polynomial $O_n(X) \in \ZZ[X]$ by
 \begin{equation*}
  O_n(X) := (X^n - 1)(X^n - X) \dots (X^n - X^{n-1}).
 \end{equation*}
 Then $O_n(p)$ is the order of $GL_n(\ZZ/p)$ for any prime $p$, and consequently we have
 \begin{equation*}
  \abs{GL_n(\ZZ/s)} = O_n(p_1) \cdot \dots \cdot O_n(p_k).
 \end{equation*}
 In order to bound the number of prime factors in this expression, we rely on the following generalisation of Dirichlet's Theorem:
 \begin{theorem}[{\cite{MR0179139}}]\label{thm_miech}
  Let $f(X) \in \ZZ[X]$ be a polynomial. Let $\rho$ and $\mu$ be natural numbers with $(\rho,\mu) = 1$.
  Then there is a constant $K > 0$ such that there are infinitely many primes $p$ with the property that
  $p \equiv \rho \text{ mod } \mu$ and the number of prime factors of $f(p)$, counted with their multiplicities,
  is bounded by $K$.
 \end{theorem}
 Applying this theorem to $O_n(X)$ with $\rho = 1$ and $\mu = o$, we obtain a positive number $K$ and an infinite set of primes $\cP$
 such that for all $p \in \cP$, we have that $p \equiv 1 \text{ mod } o$ and $O_n(p)$ has at most $K$ prime factors,
 counted with their multiplicities.
 
 Let us now pick three distinct prime numbers $p_1$, $p_2$ and $p_3$ from $\cP$ such that each of them is greater than $\nu$.
 Set $s := p_1p_2p_3$ and $r := \abs{GL_n(\ZZ/s)} \cdot s$. Then $s \equiv 1 \text{ mod } o$ and the order of $M_s$ clearly divides $r$.
 Since
 \begin{equation*}
  \abs{(\ZZ/s)^n \rtimes_{M_s} \ZZ/r} = s^n \cdot s \cdot \abs{GL_n(\ZZ/s)} =  (p_1p_2p_3)^{n+1} \cdot \abs{GL_n(\ZZ/s)},
 \end{equation*}
 our preliminary considerations apply to show that the order of $(\ZZ/s)^n \rtimes_{M_s} \ZZ/r$ contains at most
 $B := 3(n+1) + 3K$ prime factors (counted with their multiplicities).
 
 What we have to check is that every subgroup $G \leq (\ZZ/s)^n \rtimes_{M_s} \ZZ/r$ which lies in $\dress$
 has the desired properties. The proof is a direct adaptation of the arguments in \cite[Lem.\ 4.18 -- 4.21]{MR3164984}.
 
 Fix a generator $t$ of $\ZZ/r$. Then every element of $(\ZZ/s)^n \rtimes_{M_s} \ZZ/r$ can be written in the form $vt^j$
 for some $v \in (\ZZ/s)^n$ and some $j \in \NN$.
 
 Let us first consider the case of a subgroup $H \leq (\ZZ/s)^n \rtimes_{M_s} \ZZ/r$ which is cyclic mod $p$ for some prime $p$.
 Choose an extension $1 \to P \to H \to C \to 1$ such that $C$ is cyclic, $P$ is a $p$-group, and $p$ does not divide the order of $C$.
 Let $c \in C$ be a generator, and pick a preimage $vt^j$ under the epimorphism $H \to C$. Since $p$ does not divide $\abs{C}$,
 the element $d := c^{\abs{P}}$ is another generator of $C$, and $\pi((vt^j)^{\abs{P}}) = d$. Write $(vt^j)^{\abs{P}} = wt^l$.
 Set $x := (wt^l)^{[\pi(H) \colon \pi(P)]}$.
 
 Suppose that $H \cap (\ZZ/s)^n \neq P \cap (\ZZ/s)^n$. By definition, $x$ lies in the kernel of $\pi\big|_H$, which is $H \cap (\ZZ/s)^n$.
 Its image in $C$ is $d^{[\pi(H) \colon \pi(P)]} \neq 0$, so $x$ does not lie in $P \cap (\ZZ/s)^n$.
 Let $s'$ be the order of $x$. Then $s'$ divides $\abs{C}$; in particular, $p$ does not divide $s'$.
 Note that $s'$ divides $s$; write $s = \sigma \cdot s'$. By the definition of $s$, the numbers
 $\sigma$ and $s'$ are coprime. Let $k := \frac{\abs{GL_n(\ZZ/s)}}{\gcd(\abs{GL_n(\ZZ/s)},s')}$.
 Then $k\sigma$ and $s'$ are still coprime. Consequently, $(wt^l)^{k\sigma[\pi(H) \colon \pi(P)]} = x^{k\sigma} \neq 0$.
 On the other hand, one can compute as in \cite[Lem.\ 4.19]{MR3164984} that $(wt^l)^{sr'}$ is the trivial element,
 so $sr'$ does not divide $k\sigma[\pi(H) \colon \pi(P)]$.
 Dividing by $k\sigma$ on both sides, we get $s'\overline{r} \nmid [\pi(H) \colon \pi(P)]$,
 where $\overline{r}$ is some natural number containing only prime factors which are also prime factors of $s'$.
 Therefore, there is some $i \in \{ 1,2,3 \}$ and a natural number $N \geq 1$ such that $p_i \mid s'$, $p_i^N \mid s'\overline{r}$, and
 $p_i^N \nmid [\pi(H) \colon \pi(P)]$. Since $s'$ is a divisor of $\abs{H}$, the prime $p_i$ divides $\abs{H}$.
 
 It follows that $r = sr'$ is divisible by $p_i^N$. As $p$ does not divide $s'$, we must have $p \neq p_i$.
 In particular, $[\ZZ/r \colon \pi(P)]$ is also divisble by $p_i^N$.
 Since $[\pi(H) \colon \pi(P)]$ is only divisible by primes which are also prime factors of $\abs{C}$,
 the equality
 \begin{equation*}
  [\ZZ/r \colon \pi(H)] \cdot [\pi(H) \colon \pi(P)] = [\ZZ/r \colon \pi(P)]
 \end{equation*}
 implies that $p_i$ divides $[\ZZ/r \colon \pi(H)]$; in particular, $[\ZZ/r \colon \pi(H)] \geq p_i \geq \nu$.
 
 Thus, we have shown that for every subgroup $H$ which is cyclic mod $p$ for some prime $p$,
 there is an extension $1 \to P \to H \to C \to 1$ with $C$ a cyclic group, $P$ a $p$-group
 such that $p \nmid \abs{C}$, and one of the following statements is true:
 \begin{itemize}
  \item $H \cap (\ZZ/s)^n = P \cap (\ZZ/s)^n$.
  \item There is $i \in \{ 1,2,3 \}$ such that $p \neq p_i$, $p_i \mid \abs{H}$ and $p_i \mid [\ZZ/r \colon \pi(H)]$.
 \end{itemize}
 We are now going to use this to show the actual claim. So let $G \in \dress$ be a subgroup of $(\ZZ/s)^n \rtimes_{M_s} \ZZ/r$.
 Pick an extension $1 \to H \to G \to Q \to 1$ such that $H$ is cyclic mod $p$ for some prime $p$ and $Q$ is a $q$-group.
 If $p \neq q$, we may assume that $q$ does not divide $\abs{H}$.
 Note that both $[G \cap (\ZZ/s)^n \colon H \cap (\ZZ/s)^n]$ and $[\pi(G) \colon \pi(H)]$ are $q$-powers.
 Choose an extension $1 \to P \to H \to C \to 1$ with the properties we had just discussed.
 
 Assume first that $H \cap (\ZZ/s)^n = P \cap (\ZZ/s)^n$. Then $\abs{G \cap (\ZZ/s)^n} = p^kq^l$
 for some natural numbers $k$ and $l$. Choose $i$ such that $p \neq p_i \neq q$.
 Let $\gamma$ be a generator of $\ZZ/s$, and let $(g_1,\dots,g_n) \in G \cap (\ZZ/s)^n$ be
 an arbitrary element. There are natural numbers $a_j$ such that $g_j = \gamma^{a_j}$.
 It follows that $\gamma^{a_jp^kq^l} = 0$, so $a_j$ is divisible by $p_i$.
 This shows that $G \cap (\ZZ/s)^n \subset p_i(\ZZ/s)^n$. Note that by our initial choice of $p_i \in \cP$,
 it is automatically true that $p_i \geq \nu$, $p_i$ divides $s$, and $p_i \equiv 1 \text{ mod } o$.
 
 Consider now the case that $H \cap (\ZZ/s)^n \neq P \cap (\ZZ/s)^n$, so there is some $i$
 such that $p \neq p_i$, $p_i \mid \abs{H}$ and $p_i \mid [\ZZ/r \colon \pi(H)]$.
 We must also have $q \neq p_i$. Since
 \begin{equation*}
  [\ZZ/r \colon \pi(G)] \cdot [\pi(G) \colon \pi(H)] = [\ZZ/r \colon \pi(H)],
 \end{equation*}
 the prime $p_i$ must divide $[\ZZ/r \colon \pi(G)]$, so $[\ZZ/r \colon \pi(G)] \geq p_i \geq \nu$.
 This finishes the proof.
\end{proof}

\begin{proof}[Proof of Theorem \ref{thm_fhbdd.IrrSpecialAffine}] 
 We have to introduce some additional notation. Recall that we can write $\Gamma$
 as an extension $1 \to \Theta \to \Gamma \to \Delta \to 1$ of a crystallographic
 group $\Theta$ by $\Delta \in \{ \ZZ, D_\infty \}$.
 Let $A$ be the unique normal, free abelian subgroup of $\Theta$ which equals its own centraliser.
 Set $Q := \Gamma/A$. As before, set $A_s := A/sA$ for any positive integer $s$.
 Note that $A_s$ is isomorphic to $(\ZZ/s)^n$. Moreover, the virtually cyclic group $Q$
 has a normal, infinite cyclic subgroup $C \leq Q$. Let $F$ be the finite quotient $Q/C$, and denote
 for any positive integer $r$ the quotient $Q/rC$ by $Q_r$. As described in \cite[p.~359]{MR3164984},
 there is a certain semidirect product $A_s \rtimes_{\rho_{r,s}} Q_r$ whenever $r$ divides the order of
 $\text{aut}(A_s)$, and one can construct a projection map $\pi_{r,s} \colon \Gamma \to A_s \rtimes_{\rho_{r,s}} Q_r$.
 If $G \leq A_s \rtimes_{\rho_{r,s}} Q_r$ is any subgroup, write $\overline{G}$ for the preimage $\alpha_{r,s}^{-1}(\overline{G})$.
 
 First, we observe that given $\tau \in \NN$, there are natural numbers $r, s \in \NN$ such that
 \begin{enumerate}
  \item $s \equiv 1 \text{ mod } \abs{H^2(Q;A)}$.
  \item $\abs{\text{aut}(A_s)} \mid r$.
  \item For every subgroup $G \leq A_s \rtimes_{\rho_{r,s}} Q_r$ which lies in $\dress$, one of the following holds:
   \begin{itemize}
    \item The order of $H^1(pr(\overline{G});A)$ and $H^2(pr(\overline{G});A)$ is finite
     and there is $k \in \NN$ such that $k$ divides $s$, $k \geq \tau$, $k \equiv 1 \text{ mod } \abs{H^1(pr(\overline{G});A)}$,
     $k \equiv 1 \text{ mod } \abs{H^2(Q;A)}$, $k \equiv 1 \text{ mod } \abs{H^2(pr(\overline{G});A)}$, and $\overline{G} \cap A \subset kA$.
    \item $[D \colon \pi(pr(\overline{G}))] \geq \tau$.
   \end{itemize}
 \end{enumerate}
 This is non-trivial, but one can copy the proof of \cite[Lem.~4.22]{MR3164984} verbatim, noticing that the only part of the proof
 which is specific to hyperelementary subgroups is the invocation of \cite[Prop.~4.10]{MR3164984}, which we can replace by Lemma \ref{lem_IrrSpecialAffine.finitequotients}.
 Observe that
 \begin{equation*}
  \abs{A_s \rtimes_{\rho_{r,s}} Q_r} = \abs{A_s} \cdot \abs{Q_r} = \abs{(\ZZ/s)^n} \cdot \abs{\ZZ/r} \cdot \abs{F}.
 \end{equation*}
 Since we used Lemma \ref{lem_IrrSpecialAffine.finitequotients} to choose $r$ and $s$, we know that $\abs{(\ZZ/s)^n} \cdot \abs{\ZZ/r}$ contains at most $B$ prime factors.
 Moreover, the order of $F = Q/C$ does not depend on any of the choices we made, and thus always contains the same number of prime factors.
 In total, this gives us a uniform bound on the number of prime factors occurring in the order of $\abs{A_s \rtimes_{\rho_{r,s}} Q_r}$,
 counted with their multiplicites, and thus on the depth of $A_s \rtimes_{\rho_{r,s}} Q_r$.
 
 Now the proof can be finished by arguing precisely as in the proof of \cite[Prop.~4.41]{MR3164984}.
\end{proof}

\appendix
\section{On the proof of Oliver's theorem}\label{sec_oliversthm}
\renewcommand{\thesection}{A}

As promised, we are now going to review the proof of Oliver's Theorem \ref{thm_OliversTheorem} to show the existence
of the function $\bound$. The outline of the proof is basically that of \cite{MR0375361}, with some additional
input from \cite{MR520510}. However, we will deviate from the treatment in \cite{MR0375361} at the end to get a better grip
on the dimension bound.

Let $G$ be a finite group throughout. Let $\Omega(G)$ be the Burnside ring of $G$.
We think about elements in $\Omega(G)$ as equivalence classes of finite $G$-CW-complexes,
where the relevant equivalence relation $\sim_\chi$ is the following:
Two finite $G$-CW-complexes $X$ and $Y$ are $\chi$-equivalent, $X \sim_\chi Y$, if and only if
$\chi(X^H) = \chi(Y^H)$ for all subgroups $H \leq G$, where $\chi$ denotes the Euler characteristic
of a finite CW-complex. The disjoint union and product operations induce the ring structure on $\Omega(G)$.
See \cite[p.~90]{MR0423390} and \cite{MR0394711} for more information on this description of $\Omega(G)$.

Let $\Delta(G) \subset \Omega(G)$ be the subset given by
\begin{equation*}
 \begin{split}
  \Delta(G) := \{ x \in \Omega(G) \mid &\text{ There is a finite contractible $G$-CW-complex $X$} \\
   &\qquad\text{with $x = [X] - 1$.} \}.
 \end{split}
\end{equation*}
As Oliver observed in \cite[p.~90]{MR0423390}, this is an ideal in $\Omega(G)$.
Let $\mathsf{gh}_G \colon \Omega(G) \to \ZZ$ be the ``ghost map'' that sends $[X]$ to $\chi(X^G)$.
Then the image of $\Delta(G)$ under $\mathsf{gh}_G$ is an ideal in $\ZZ$;
we let $n_G$ denote the unique non-negative generator.
One easily observes that $n_G = 1$ if there is a finite contractible $G$-CW-complex without a global fixed point.

Another important concept is that of a resolving function:

\begin{definition}[{\cite[bottom of p.~159]{MR0375361}}]
 Let $\cS(G)$ denote the poset of subgroups of $G$. A function $\phi \colon \cS(G) \to \ZZ$
 is a \emph{resolving function} if the following holds:
 \begin{itemize}
  \item $\phi$ is constant on conjugacy classes of subgroups.
  \item For all $H \leq G$, the order of the Weyl group $[N_G(H) \colon H]$ divides $\phi(H)$.
  \item If $H \in \cyc_p$ for some prime $p$, then $\sum_{K \supset H} \phi(K) = 0$.
 \end{itemize}
\end{definition}

Every finite, contractible $G$-CW-complex $X$ gives rise to a resolving function $\phi_X$ \cite[Prop.~2 \& Lem.~2]{MR0375361}.
The set
\begin{equation*}
 \{ \phi(G) \mid \phi \text{ is a resolving function for $G$} \} \subset \ZZ
\end{equation*}
forms a subgroup, and we let $r_G$ denote the unique non-negative generator of this group.
If $[X] - 1$ is a preimage of $n_G$ with respect to $\mathsf{gh}_G$, then $n_G = \chi(X^G) - 1 = \phi_X(G)$.
It follows that $r_G$ is always a divisor of $n_G$.

\begin{theorem}[{cf.\ \cite[Thm.~2]{MR0375361}}]\label{thm_OliversTheorem.prelim}
 If $r_G = 1$ and $G$ is not a $p$-group for any prime $p$, then there is a finite, contractible $G$-CW-complex $X$
 without global fixed point whose dimension is bounded by $4 \cdot \depth(G) + 2$. 
\end{theorem}

As a first step towards Theorem \ref{thm_OliversTheorem.prelim}, one constructs a \emph{$G$-resolution} $Y$,
i.e., a finite, $n$-dimensional and $(n-1)$-connected $G$-CW-complex $Y$ with $Y^G = \emptyset$
such that $H_n(Y;\ZZ)$ is a finitely generated projective $\ZZ[G]$-module.
In addition, we will see that the dimension $n$ of $Y$ can be bounded by $2 \cdot \depth(G)$.

To keep track of how the construction proceeds, we try to make the induction as explicit as possible.
Let $\cS(G)$ be again the poset (with respect to $\supseteq$) of subgroups of $G$.
Define the \emph{rank} of a subgroup $H \leq G$ to be
\begin{equation*}
 \rank(H) := \max \{ \rank(K) \mid K \in \cS(G), K \supsetneq H \} + 1,
\end{equation*}
where we let $\max \emptyset = 0$. Note that the depth of $G$ is precisely $\rank(\{1_G\})$,
and that conjugate subgroups have equal rank.
For each $r \leq \depth(G)$, choose a linear order $\preccurlyeq_r$ on the set of subgroups of rank $r$.
Then
\begin{equation*}
 H \preccurlyeq H' :\Longleftrightarrow \rank(H) < \rank(H') \text{ or } (\rank(H) = \rank(H') \text{ and } H \preccurlyeq_{\rank(H)} H')
\end{equation*}
defines a linear order on $\cS(G)$. We write $H \prec H'$ if $H \preccurlyeq H'$ and $H \neq H'$.

The construction of $Y$ proceeds inductively along the finite linear order $(\cS(G) \setminus \{1\}, \prec)$.
By assumption, we may choose a resolving function $\phi$ on $G$ with $\phi(G) = -1$.
Then we claim that for every $H \in \cS(G) \setminus \{ 1 \}$, there is a finite $G$-CW-complex $Y(H)$
such that the following holds:
\begin{enumerate}
 \item The complex $Y(H)$ contains only cells of type $G/K$ for $K \neq G$ and $K \preccurlyeq H$;
  in particular, it has no global fixed point.
 \item The dimension of $Y(H)$ is bounded by $2 \cdot \rank(H)$.
 \item For every $K \preccurlyeq H$, the dimension of the $K$-fixed point set $Y(H)^K$ is bounded by $2 \cdot \rank(K)$.
 \item\label{cond_eulerchar} $\chi(Y(H)^K) = 1 + \sum_{K' \supseteq K} \phi(K')$ for all $K \preccurlyeq H$.
 \item If $K \preccurlyeq H$ is a $p$-group for some prime $p$, the $K$-fixed points $Y(H)^K$ are $\ZZ/p$-acyclic.
\end{enumerate}
For the start of the induction, we can set $Y(G) := \emptyset$. Suppose that $Y(H^-)$ has been constructed for some subgroup $H^- \leq G$.
Let $H$ be the immediate successor of $H^-$. If there is some conjugate $gHg^{-1}$ of $H$ such that $gHg^{-1} \prec H$,
we may set $Y(H) := Y(H^-)$. Suppose that no conjugate of $H$ is $\prec$-smaller than $H$.
Then there is some $n \in \ZZ$ such that $\chi(Y(H^-)^H) + n \cdot [N_G(H) \colon H] = 1 + \sum_{K \supset H} \phi(K)$.

Assume first that $H$ is no $p$-group. If $n = 0$, set $Y(H) := Y(H^-)$. Otherwise, we can add cells of type $G/H$
in sufficiently low dimensions to enforce condition \ref{cond_eulerchar} without affecting the other properties.

So suppose now that $H$ is a $p$-group for some prime $p$. Attach successively cells of type $G/H$ to $Y(H^-)$
to construct a $G$-CW-complex $Y'$ whose $H$-fixed point set $(Y')^H$ has dimension bigger than the $K$-fixed point sets
$(Y')^K$ for all $K \supsetneq H$ and is $(\dim(Y')-1)$-connected. We can arrange that $\dim(Y') \leq 2 \cdot \rank(H) - 1$.
By the argument on p.~162 of \cite{MR0375361}, the top-dimensional homology $H_{\dim(Y')}(Y;\ZZ/p)$ is free.
This allows us to glue on a set of $(\dim(Y')+1)$-cells of type $G/H$ to obtain a finite $G$-CW-complex $Y(H)$
whose dimension is bounded by $2 \cdot \rank(H)$ and which is $\ZZ/p$-acyclic. One checks that $Y(H)$ has
all other desired properties.

At the end of the induction, we have a finite $G$-CW-complex $Y_0$ which has no global fixed point, whose dimension $n'$
is bounded by $2 \cdot \depth(G) - 2$, whose fixed-point sets under non-trivial $p$-groups are $\ZZ/p$-acyclic and
which satisfies $\chi(Y_0^H) = 1 + \sum_{K \supseteq H} \phi(K)$ for all $H \neq 1$.

By another induction along the skeleta, we can glue on free $G$-cells to produce an $(n'+1)$-dimensional and $n'$-connected $G$-CW-complex
$Y$ which has no global fixed point and whose top-dimensional homology $H_{n'+1}(Y;\ZZ)$ is finitely generated and projective
as a $\ZZ[G]$-module (see \cite[Proof of Thm.~2]{MR0375361} for the last claim). Setting $n := n' + 1$, we have found a $G$-resolution.

Theorem \ref{thm_OliversTheorem.prelim} can be derived from the existence of a $G$-resolution $Y$ as follows: Take the join $X' := Y * Y$.
We will think about the join of two spaces $Z$ and $Z'$ as $(C(Z) \times Z') \cup_{Z \times Z'} (Z \times C(Z')$
via the homeomorphism
\begin{equation*}
 \begin{split}
  Z * Z' &\to (C(Z) \times Z') \cup_{Z \times Z'} (Z \times C(Z') \\
  (z,t,z') &\mapsto \begin{cases}
                     (z,(1-2t,z')) \in Z \times C(Z') & t \leq \frac{1}{2}, \\
                     ((z,2t-1),z') \in C(Z') \times Z & t \geq \frac{1}{2}.
                    \end{cases}
 \end{split}
\end{equation*}
This description makes it obvious that $X'$ is a $G$-CW-complex without a global fixed point whose dimension can be bounded by
$2n + 1 \leq 4 \cdot \depth(G) + 1$. Moreover, we can use the given decomposition to apply the Seifert--van Kampen Theorem,
and then proceed by induction with the Hurewicz Theorem and Mayer--Vietoris sequence for homology to show that $X'$ is $2n$-connected.
The isomorphisms
\begin{equation*}
 H_{2n+1}(X';\ZZ) \xrightarrow{\cong} H_{2n}(Y \times Y;\ZZ) \xleftarrow{\cong} H_n(X';\ZZ) \otimes_\ZZ H_n(X';\ZZ),
\end{equation*}
which can be assembled from the Mayer--Vietoris sequence and the K\"unneth Theorem, are both $\ZZ[G]$-linear.
We can now invoke \cite[Prop.~C.3]{MR1961198} to deduce that $H_{2n+1}(X',\ZZ)$ is a stably free $\ZZ[G]$-module.
Hence, we can add free $(2n+1)$- and $(2n+2)$-cells to produce a finite, contractible $G$-CW-complex $X$
without global fixed point whose dimension is bounded by $4 \cdot \depth(G) + 2$.

\begin{corollary}
 Suppose $G$ is not a $p$-group for any prime $p$. Then the following are equivalent:
 \begin{enumerate}
  \item $r_G = 1$.
  \item There is a finite, contractible $G$-CW-complex $X$ with $X^G = \emptyset$ whose dimension is bounded by $4 \cdot \depth(G) + 2$.
  \item $n_G = 1$.
 \end{enumerate}
\end{corollary}

Since Oliver has shown in \cite[Thm.~5]{MR0375361} that $r_G = 1$ if and only if $G \notin \dress$, Theorem \ref{thm_OliversTheorem} follows
from the well-known fact that every finite $G$-CW-complex is $G$-homotopy equivalent to a finite \gscplx{G} of
equal dimension, see e.g.~\cite[Prop.~A.4]{MR1961198}.

\bibliographystyle{alpha}
\bibliography{../Literatur}

\end{document}